\documentclass[a4paper,reqno,11pt]{amsart}%
\usepackage{amssymb}
\usepackage{amsfonts}
\usepackage{amsmath}
\usepackage[left=3.2cm,right=3.2cm,bottom=3cm,top=3cm]{geometry}
\usepackage{graphicx}%
\usepackage{amssymb,latexsym,amsmath,amsthm,amscd}
\providecommand{\U}[1]{\protect\rule{.1in}{.1in}}

\makeatletter
\@namedef{subjclassname@2010}{  \textup{2010} Mathematics Subject Classification}
\makeatother
\newtheorem{theorem}{Theorem}[section]
\theoremstyle{plain}

\newtheorem{corollary}[theorem]{Corollary}

\newtheorem{example}[theorem]{Example}

\newtheorem{fact}[theorem]{Fact}

\newtheorem{proposition}[theorem]{Proposition}

\numberwithin{equation}{section}

\theoremstyle{definition}
\newtheorem{convention}[theorem]{Convention}
\newtheorem{definition}[theorem]{Definition}
\newtheorem{remark}[theorem]{Remark}

\newcommand {\ggs}{\gtrapprox}
\newcommand {\n}{\mathbb{N}}
\newcommand {\starn}{\,^*\mathbb{N}}

\newcommand {\ld}{\underline{d}}
\newcommand {\ud}{\overline{d}}

\newcommand {\uld}{\overline{ld}}
\newcommand {\lld}{\underline{ld}}

\newcommand {\st}{\operatorname{st}}

\def \BD{\operatorname{BD}}
\def \lBD{\operatorname{\ell BD}}
\def \r{\mathbb{R}}

\begin{document}
\baselineskip=18pt

\title[Approximate Polynomial Structure in Additively Large Sets]{Approximate Polynomial
Structure\\ in Additively Large Sets}
\author[Di Nasso et. al.]{Mauro Di Nasso, Isaac Goldbring, Renling Jin,
Steven Leth, Martino Lupini, Karl Mahlburg}
\thanks{The authors were supported in part by the American Institute of Mathematics through its SQuaREs program.  I. Goldbring was partially supported by NSF CAREER grant DMS-1349399.  M. Lupini was supported by the York University Susan Mann Dissertation Scholarship.  K. Mahlburg was supported by NSF Grant DMS-1201435.}
\address{Dipartimento di Matematica, Universita' di Pisa, Largo Bruno
Pontecorvo 5, Pisa 56127, Italy}
\email{dinasso@dm.unipi.it}
\address{Department of Mathematics, Statistics, and Computer Science,
University of Illinois at Chicago, Science and Engineering Offices M/C 249,
851 S. Morgan St., Chicago, IL, 60607-7045}
\email{isaac@math.uic.edu}
\address{Department of Mathematics, College of Charleston, Charleston, SC,
29424}
\email{JinR@cofc.edu}
\address{School of Mathematical Sciences, University of Northern Colorado,
Campus Box 122, 510 20th Street, Greeley, CO 80639}
\email{Steven.Leth@unco.edu}
\address{Fakult\"{a}t f\"{u}r Mathematik, Universit\"{a}t Wien,
Oskar-Morgenstern-Platz 1, Room 02.126, 1090 Wien, Austria.}
\email{martino.lupini@univie.ac.at}
\address{Department of Mathematics, Louisiana State University, 228 Lockett
Hall, Baton Rouge, LA 70803}
\email{mahlburg@math.lsu.edu}
\thanks{}
\date{}
\keywords{nonstandard analysis, log density}
\subjclass[2010]{}
\dedicatory{ }

\begin{abstract}
We show that any subset of the natural numbers with positive logarithmic Banach density contains a set that is within a factor of two of a
geometric progression, improving the bound on a previous result of the
authors. \ Density conditions on subsets of the natural numbers that imply the
existence of approximate powers of arithmetic progressions are developed and explored.
\end{abstract}

\maketitle

\section{Introduction\label{Section: Introduction and Preliminaries}}

In \cite{DGJLLM}, the authors introduced a measure space, obtained by
taking a quotient of a Loeb measure space, that has the property that
multiplication is measure-preserving and for which standard sets of positive logarithmic
density have positive measure. \ The log Banach density of a standard set (see
Section 2 below for the definition) was also introduced, and this measure
space framework was used, in conjunction with Furstenberg's Recurrence
Theorem, to obtain a standard result about the existence of approximate
geometric progressions in sets of positive log Banach density. \ In this paper,
we improve the bounds of approximation of this result by using Szemer\'{e}di's
Theorem together with a \textquotedblleft logarithmic change of
coordinates.\textquotedblright\ \ More specifically, in Proposition 3.1, we
show that if $A$ is a standard subset of the natural numbers, then the Banach
density of $\left\{  \left\lceil \log_{2}(x)\right\rceil :x\in A\right\}  $ is
greater than or equal to the log Banach density of $A$. \ This allows us to
use Szemer\'{e}di's Theorem to show that every set of positive Banach log
density contains a set which is \textquotedblleft within a factor of
2\textquotedblright\ of being a geometric sequence; Theorem 3.3 provides a
precise version of this statement. \ We also explore a family of densities on the natural
numbers, the (upper) $r$-Banach densities for $0<r\leq1$, which have the
property that positive $1/m$-Banach density implies the existence of
approximate $m$th powers of arithmetic progressions, in a sense made precise
in Theorem 3.7.  (This family of densities was introduced in \cite{DGJLLM}, although $BD_{m}(A)$ in that paper corresponds to $BD_{1/m}(A)$ here.) \ 

In Section 2 we establish some properties of the log Banach density and the
$r$-Banach densities, most notably that the log Banach density of a set
$A$ is always less than or equal to every $r$-Banach density of $A$, and that if
$r<s$ then the $r$-Banach density of $A$ is less than or equal to the
$s$-Banach density of $A$ (Theorem 2.12). \ These inequalities can both be strict.
\ In fact, it is easy to see that the log Banach density of a set $A$ could be 0 while every
$r$-Banach density of $A$ is 1, and in Example 2.13 we see that if $r<s$ then it is possible to
have the $r$-Banach density of a set $A$ be 0 while the $s$-Banach density of $A$ is 1.

In Section 3 we establish the aforementioned approximation results and
provide examples to show that the level of approximation is optimal.

We use nonstandard methods, which simplifies a number of the arguments.\ \ For
an introduction to nonstandard methods aimed specifically toward applications
to combinatorial number theory, see \cite{jin}.

\subsection{Acknowledgements}

This work was initiated during a week-long meeting at the American
Institute for Mathematics on August 4-8, 2014 as part of the SQuaRE (Structured
Quartet Research Ensemble) project \textquotedblleft Nonstandard Methods in
Number Theory.\textquotedblright\ The authors would like to thank the
Institute for the opportunity and for the Institute's hospitality during their stay.

\section{$r$-density and logarithmic density
\label{Section: $r$-density and logarithmic density}}

\begin{convention}
In this paper, $\n$ denotes the set of \emph{positive} natural numbers. For any real numbers
$a\leq b$, we set $[a,b]:=\{x\in \n \ : \ a\leq x \leq b\}$.  We make a similar convention for the intervals $(a,b]$, $[a,b)$, and $(a,b)$.
\end{convention}

We recall some well-known densities on $\n$.

\begin{definition}
Suppose that $A\subseteq \n$ and $0<r\leq 1$.
\begin{itemize}
\item The \emph{upper $r$-density of $A$} is defined to be
$$\ud_r(A):=\limsup_{n\rightarrow\infty}\frac{r}{n^{r}}\sum_{x\in A\cap
[1,n]}\frac{1}{x^{1-r}}.$$
\item The \emph{lower $r$-density of $A$} is defined to be
$$\ld_r(A):=\liminf_{n\rightarrow\infty}\frac{r}{n^{r}}\sum_{x\in A\cap [1,n]}\frac{1}{x^{1-r}}.$$
\end{itemize}
\end{definition}

Note that $\ud_1(A)$ and $\ld_1(A)$ are simply the usual upper and lower asymptotic densities
of $A$, respectively. For that reason, we omit the subscript $r$ when $r=1$.



\begin{definition}
Suppose that $A\subseteq \n$.  Then:
\begin{itemize}
\item The \emph{upper logarithmic density of $A$} is defined to be
$$\uld(A):=\limsup_{n\rightarrow\infty}\frac{1}{\ln n}\sum_{x\in A\cap [1,n]}\frac{1}{x}.$$
\item The \emph{lower logarithmic density of $A$} is defined to be
$$\lld(A):=\liminf_{n\rightarrow\infty}\frac{1}{\ln n}\sum_{x\in A\cap [1,n]}\frac{1}{x}.$$
\end{itemize}
\end{definition}

%

The following result establishing relationships amongst the above densities was proven in \cite{raja}.

\begin{fact}\label{compare}
For $A\subseteq \n$ and $0<r<s\leq 1$, we have $$\ld_{s}(A)\leq\ld_{r}(A)\leq\lld(A)\leq
\uld(A)\leq\ud_{r}(A)\leq\ud_{s}(A).$$
\end{fact}

In working with these densities, we often use the following elementary estimates (established using an integral approximation):  for any $a<b$ in $\n$, we have
$$\sum_{x=a}^{b-1}\frac{1}{x^{1-r}}\leq \frac{b^r-a^r}{r}\leq \sum_{x=a+1}^b \frac{1}{x^{1-r}}.$$
\begin{theorem}\label{positiveRdensity}
For $A\subseteq \n$ and $0<r\leq 1$, we have
$$\ud_{r}(A)\geq 1-\left(1-\ud(A)\right)^{r}.$$
\end{theorem}

\begin{proof}
Set $\alpha:=\ud(A)$ and take $H\in {}^\ast \n\setminus \n$ such that $\frac{N}{H}\approx \alpha$, where $N:=\left \vert {}^*A\cap [1,H]\right\vert$.  Set $\epsilon:=\frac{N}{H}-\alpha$, so $\epsilon$ is a (possibly negative) infinitesimal.  We now have
\begin{alignat}{2}
\ud_{r}(A)&\geq \st\left(\frac{r}{H^{r}}\sum_{x\in\,^*\!A\cap[1,H]}\frac{1}{x^{1-r}}\right)\notag \\ \notag
               &\geq \st\left(\frac{r}{H^{r}}\sum_{x\in\,(H-N,H]}\frac{1}{x^{1-r}}\right) \\ \notag
                       &\geq \st\left(\frac{r}{H^r}\cdot \frac{H^r-(H-N)^r}{r}\right)\\ \notag
                       &=\st(1-(1-(\alpha+\epsilon))^r)\\ \notag
                       &=1-(1-\alpha)^r.\notag \qedhere
\end{alignat}
\end{proof}

\medskip

\begin{corollary}\label{positiveimpliespositive}
If $\ud(A)>0$, then $\ud_r(A)>0$ for all $0<r\leq 1$.
\end{corollary}


\begin{remark}
It is easy to construct a set $A\subseteq\n$ with $\ud(A)=1$ and $\uld(A)=0$.
As a consequence of the theorem above, we also have $\ud_r(A)=1$ for any $0<r\leq 1$.
\end{remark}

We now introduce the corresponding uniform versions of the above densities.
\begin{definition}
For $A\subseteq \n$ and $0<r\leq 1$, the \emph{(upper) $r$-Banach density}
of $A$ is defined to be
$$\BD_{r}(A):=\lim_{n\to\infty}\sup_{k\in\n}\frac{r}{n}
\sum_{x\in A\cap [k,(k^{r}+n)^{1/r}]}\frac{1}{x^{1-r}}.$$
\end{definition}

Note that $\BD(A)=\BD_1(A)$ is the usual upper Banach density of $A$. Note also that
$\ud_r(A)\leq\BD_r(A)$ from definition.

\begin{definition}
For $A\subseteq \n$, the \emph{(upper) log Banach density} of $A$ is
$$\lBD(A):=\lim_{n\to \infty}\sup_{k\geq 1}\frac{1}{\ln n}\sum_{x\in A\cap [k,nk]}\frac{1}{x}.$$
\end{definition}

Of course one could also define the lower $r$-Banach density and
the lower log Banach density,
but in this paper we only focus on the upper $r$-Banach density and upper log Banach density.

The following nonstandard formulation of
$r$-Banach density and log Banach density follows immediately from the nonstandard characterization of limit.

\begin{proposition}\label{nsubld}
 Let $A\subseteq \n$, $0<r\leq 1$, and $0\leq\alpha\leq 1$.
 \begin{enumerate}
 \item $\BD_{r}(A)\geq\alpha$ if and only
 if there are $k,N\in \starn$ with $N>\n$ such that $$\st\left(\frac{r}{N}
\sum_{x\in A\cap [k,(k^{r}+N)^{1/r}]}\frac{1}{x^{1-r}}\right)\geq\alpha.$$
 \item $\lBD(A)\geq\alpha$ if and only if there are $k,N\in \starn$ with $N>\n$ such that
 $$\st\left(\frac{1}{\ln N}\sum_{x\in {}^{\ast }\!{A}\cap [k,Nk]}
 \frac{1}{x}\right)\geq\alpha.$$

 \end{enumerate}
\end{proposition}

We now establish the uniform version of Fact \ref{compare} above.  The results in \cite{raja} do not immediately apply in the uniform setting.  Nevertheless, our proof is inspired by the arguments from \cite{raja}, although we argue in the nonstandard model to make the idea more transparent.

\begin{theorem}
For any $A\subseteq\n$ and $0<r<s\leq 1$, we have
$$\lBD(A)\leq \BD_{r}(A)\leq\BD_{s}(A).$$
\end{theorem}

\begin{proof}
We first prove that $\BD_r(A)\leq\BD_s(A)$.

Let $0<\alpha<1$ be such that $\BD_s(A)<\alpha$.
It suffices to show that $\beta:=\BD_r(A)\leq\alpha$.
By Proposition \ref{nsubld}, we can find $a,b\in\,^*\n$
such that $b^r-a^r>\n$ and
$$\BD_r(A)=\st\left(\left(\sum_{n=a}^b\frac{\chi_A(n)}{n^{1-r}}\right)\left(
\sum_{n=a}^b\frac{1}{n^{1-r}}\right)^{-1}\right).$$
Here, $\chi_A$ denotes the characteristic function of (the nonstandard extension of) $A$.  Next note that if $c,d\in\,^*\n$ are such that $d^s-c^s>\n$,
then Proposition \ref{nsubld} once again implies that
$$\st\left(\left(\sum_{i=c}^{d}\frac{\chi_A(i)}{i^{1-s}}\right)
\left(\sum_{i=c}^{d}\frac{1}{i^{1-s}}\right)^{-1}\right)\leq\BD_s(A)<\alpha.$$
Choose $m\in [a,b]$ such that $m^r-a^r>\n$ and
$$\left(\sum_{n=a}^m\frac{1}{n^{1-r}}\right)
\left(\sum_{n=a}^b\frac{1}{n^{1-r}}\right)^{-1}\approx 0.$$
(For example, let $m=\lceil ((b^r-a^r)^{1/2}+a^r)^{1/r}\rceil$.)
Since $x\mapsto x^s-x^r$ is an increasing function, we have that $m^s-a^s\geq m^r-a^r$. Hence $\displaystyle
\sum_{n=a}^i\frac{\chi_A(n)}{n^{1-s}}<
\alpha\sum_{n=a}^i\frac{1}{n^{1-s}}$ for any $i>m$.
Now we have
\allowdisplaybreaks{
\begin{align*}
\sum_{n=a}^b\frac{\chi_A(n)}{n^{1-r}} &
=\sum_{n=a}^b\frac{\chi_A(n)}{n^{1-s}}\frac{1}{n^{s-r}}\\
 &=\sum_{n=a}^{b}\frac{\chi_A(n)}{n^{1-s}}\left(\sum_{i=n}^{b}\left(\frac{1}{i^{s-r}}
-\frac{1}{(i+1)^{s-r}}\right)
+\frac{1}{(b+1)^{s-r}}\right)\\
 &=\sum_{n=a}^b\sum_{i=n}^{b}\frac{\chi_A(n)}{n^{1-s}}\left(\frac{1}{i^{s-r}}
-\frac{1}{(i+1)^{s-r}}\right)
+\sum_{n=a}^b\frac{\chi_A(n)}{n^{1-s}}\frac{1}{(b+1)^{s-r}}\\
 &=\sum_{i=a}^b\sum_{n=a}^i\frac{\chi_A(n)}{n^{1-s}}\left(\frac{1}{i^{s-r}}
-\frac{1}{(i+1)^{s-r}}\right)
+\sum_{n=a}^b\frac{\chi_A(n)}{n^{1-s}}\frac{1}{(b+1)^{s-r}}\\
 &<\alpha\sum_{i=a}^b\sum_{n=a}^i\frac{1}{n^{1-s}}
\left(\frac{1}{i^{s-r}}-\frac{1}{(i+1)^{s-r}}\right)
+\alpha\sum_{n=a}^b\frac{1}{n^{1-s}}\frac{1}{(b+1)^{s-r}}\\
 &\quad +(1-\alpha)\sum_{i=a}^m\sum_{n=a}^i\frac{1}{n^{1-s}}
\left(\frac{1}{i^{s-r}}-\frac{1}{(i+1)^{s-r}}\right)\\
 &=\alpha\sum_{n=a}^b\frac{1}{n^{1-s}}\sum_{i=n}^{b}
\left(\frac{1}{i^{s-r}}-\frac{1}{(i+1)^{s-r}}\right)
+\alpha\sum_{n=a}^b\frac{1}{n^{1-s}}\frac{1}{(b+1)^{s-r}}\\
 &\quad +(1-\alpha)\sum_{n=a}^m\frac{1}{n^{1-s}}\sum_{i=n}^{m}
\left(\frac{1}{i^{s-r}}-\frac{1}{(i+1)^{s-r}}\right)\\
 &=\alpha\sum_{n=a}^b\frac{1}{n^{1-s}}\frac{1}{n^{s-r}}\\
 &\quad +(1-\alpha)\sum_{n=a}^m\frac{1}{n^{1-s}}\frac{1}{n^{s-r}}-
(1-\alpha)\sum_{n=a}^m\frac{1}{n^{1-s}}\frac{1}{(m+1)^{s-r}}\\
 &\leq\alpha\sum_{n=a}^b\frac{1}{n^{1-r}}
+(1-\alpha)\sum_{n=a}^m\frac{1}{n^{1-r}}.
\end{align*}
}
We conclude that $\beta\leq\alpha$ since
$\displaystyle\left(\sum_{n=a}^m\frac{1}{n^{1-r}}\right)
\left(\sum_{n=a}^b\frac{1}{n^{1-r}}\right)^{-1}\approx 0$.

In the arguments above, if we let $r=0$ and instead require that $\ln(m)-\ln(a)>\n$, then we get $\lBD(A)\leq\BD_s(A)$.
\end{proof}

It is easy to see that the set $A=\bigcup_{n=1}^{\infty}[n!,2n!]$  has the
property that $\lBD(A)=0$ while $\BD_r(A)=1$ for every $r\in(0,1]$. \ The
following example shows that the $r$-Banach densities can also disagree to
this extent.

\begin{example}
For any $0<r<s\leq 1$ there is a set $A\subseteq\n$ such that $\BD_r(A)=0$ and $\BD_s(A)=1$.
\end{example}

\begin{proof}
Let $(a_n)$ be a sequence of positive integers defined by setting $a_1$ to be any integer larger than $1$ and $a_{n+1}:=a_n^2$. Let
\[A=\bigcup_{n=1}^{\infty}\left[a_n^{1/(rs)},\left(a_n^{1/s}+1\right)^{1/r}\right].\]
We show that $\BD_r(A)=0$ and $\BD_s(A)=1$.

Suppose that $k,N\in {}^*\n$ with $N>\n$ are such that $$\BD_r(A)=\st\left(\frac{r}{N}\sum_{x\in {}^*A\cap [k,(k^r+N)^{1/r}]}\frac{1}{x^{1-r}}\right).$$  Let $\nu$ be the maximal $m\in {}^*\n$ such that $$[a_m^{1/rs},(a_m^{1/s}+1)^{1/r}\cap [k,(k^r+N)^{1/r}]\not=\emptyset.$$  Note then that $${}^*A\cap [k,(k^r+N)^{1/r}]\subseteq [k,(\sqrt{a_\nu}^{1/s}+1)^{1/r}]\cup [a_\nu^{1/rs},(a_\nu^{1/s}+1)^{1/r}].$$  The latter interval is negligble:
\begin{eqnarray*}
\lefteqn{\st\left(  \frac{r}{N}\sum_{x\in
\,[a_{\nu}^{1/(rs)},\left(  a_{\nu}^{1/s}+1\right)  ^{1/r}]}\frac{1}{x^{1-r}%
}\right)]}\\& &=\st\left(  \frac{r}{N}\left(  \left(  \frac{\left(  \left(  a_{\nu}%
^{1/s}+1\right)  ^{1/r}\right)  ^{r}}{r}\right)  -\left(  \frac{\left(
\left(  a_{\nu}^{1/s}\right)  ^{1/r}\right)  ^{r}}{r}\right)  \right)
\right)\\  & &=\st\left(  \frac{1}{N}\right)\\  & &=0.
\end{eqnarray*}

Next observe that $(\sqrt{a_\nu}^{1/s}+1)^{1/r}<2(\sqrt{a_\nu})^{1/rs}\leq 2\sqrt{(k^r+N)^{1/r})}$.  If $2\sqrt{(k^r+N)^{1/r})}<k$, then the above computation shows that $\BD_r(A)=0$.  Thus, we may assume that $2\sqrt{(k^r+N)^{1/r})}\geq k$, from which it is readily verified that $N>k^r$.  It follows that

\begin{align*}
\BD_r(A)&\leq \st\left(  \frac{r}{N}\sum_{x\in\,[k,2(\sqrt{k^{r}+N})^{1/r}]}\frac{1}{x^{1-r}%
}\right)\\    & =\st\left(  \frac{r}{N}\left(  \frac{2^{r}(\sqrt{k^{r}+N})}%
{r}-\frac{k^{r}}{r}\right)  \right)  \\
& \leq \st\left(  \frac{r}{N}\left(  \frac{2^{r}(\sqrt{N+N})}{r}\right)
\right)\\  &=0.
\end{align*}


For showing $\BD_s(A)=1$,
it suffices to show that $\left(\left(a_n^{1/s}+1\right)^{1/r}\right)^{s}
-a_n^{1/r}>\n$ when $n>\n$.  Indeed, if $N\in {}^*\n\setminus \n$ is such that $\left(\left(a_n^{1/s}+1\right)^{1/r}\right)^{s}
-a_n^{1/r}>N$, then ${}^*A$ contains the interval $$[a_n^{1/rs},((a_n^{1/rs})^s+N)^{1/s}].$$ Note that
\begin{eqnarray*}
\lefteqn{\left(\left(a_n^{1/s}+1\right)^{1/r}\right)^{s}-a_n^{1/r}}\\
& &=\left(a_n^{1/s}+1\right)\left(a_n^{1/s}+1\right)^{(s-r)/r}-a_n^{1/r}\\
& &\geq\left(a_n^{1/s}+1\right)a_n^{(s-r)/(rs)}-a_n^{1/r}\\
& &=a_n^{\frac{1}{s}+\frac{s-r}{rs}}+a_n^{\frac{s-r}{rs}}-a_n^{\frac{1}{r}}
=a_n^{\frac{s-r}{rs}}.
\end{eqnarray*}
It remains to observe that $a_n^{\frac{s-r}{rs}}>\n$ because $a_n>\n$
and $(s-r)/(rs)$ is a positive standard real number.
\end{proof}

\section{Polynomial structure and multiplicative structure}

In what follows, $\log$ denotes $\log_2$. For $A\subseteq\n$, set 
\[\log A:=\{\lceil\log x\rceil:x\in A\}.\]

We also introduce some convenient notation:  for $k,N\in \starn$ and $E\subseteq \starn$, set $L_{k,N}(E)=\frac{1}{\ln N}\sum_{x\in E\cap [k,Nk]} 1/x$.
\begin{proposition}\label{logA}
If $A\subseteq\n$, we have $\BD(\log A)\geq\lBD(A)$.
\end{proposition}

\begin{proof}

Without loss of generality, we can assume that $\lBD(A)=\alpha>0$.  Take $k,N\in\starn$ with $N>\n$ so that
$\st(L_{k,N}({}^\ast A))=\alpha$. We first claim that we can assume that
$k$ and $kN$ are integer powers of $2$.  Indeed, choose integers $a,b$ such that $2^{a-1}<k\leqslant 2^a$ and
$2^b\leqslant kN<2^{b+1}$. Note that $b-a>\n$.  Observe now that
$$\sum_{x=k}^{2^a-1}1/x,\sum_{x=2^b+1}^{Nk}1/x\leq \ln 2,$$ so 
$$L_{k,N}({}^\ast A)\approx \frac{1}{\ln N}\sum_{x\in {}^\ast A\cap [2^a,2^b]}1/x.$$
It remains now to notice that $\ln(2^{b-1})\leq \ln N\leq \ln (2^{b-a})+\ln 2$, whence   
$$\frac{1}{\ln N}\sum_{x\in {}^\ast A\cap [2^a,2^b]}1/x\approx \frac{1}{\ln 2^{b-a}}\sum_{x\in {}^\ast A\cap [2^a,2^b]}1/x.$$

In light of the previous paragraph, we may take $a<b$ in $\starn$ so that $\st(L_{2^a,2^{b-a}}({}^\ast A))=\alpha$.  For $a\leq i<b$, set $I_i:=[2^i+1,2^{i+1}]$.  Observe that $\lceil \log(x)\rceil=i+1$ for all $x\in I_i$.  Set $\mathcal{I}:=\{i \ : \ I_i\cap {}^\ast A\not=\emptyset\}$.  We then have:

\begin{alignat}{2}
\left\vert \log({}^\ast A)\cap(a,b]\right\vert&=\left\vert \mathcal{I}\right\vert \notag \\ \notag
                                                               &=\sum_{i\in \mathcal{I}} \log(2^{i+1})-\log(2^i)\notag \\
                                                               &\geq \log(e)\sum_{i\in \mathcal{I}}\sum_{x\in {}^\ast A\cap [2^i,2^{i+1})}\frac{1}{x}.\notag                                                            
\end{alignat}

Recalling that $\ln(2^{b-a})=\frac{b-a}{\log(e)}$, it follows that
$$\BD(\log(A))\ggs \frac{\left\vert \log({}^\ast A)\cap(a,b]\right\vert}{b-a}\ggs \frac{1}{\ln 2^{b-a}}\sum_{x\in {}^\ast A\cap [2^a,2^b-1)}\frac{1}{x}=\alpha.$$
\end{proof}

We now come to the central notion of this paper.

\begin{definition}
Fix $c,r\in \r^{>0}$.  
\begin{enumerate}
\item For $a,x\in \r$, we say that $a$ is a \emph{$(c,r)$-approximation of $x$} if $a\in [x,x+cx^r)$.
\item For $A,X\subseteq \r$, we say that $A$ is a \emph{$(c,r)$-approximate subset of $X$} if every $a\in A$ is an $(c,r)$-approximation
of some $x\in X$.
\end{enumerate}
\end{definition}

\begin{theorem}\label{geometricsequence}
Suppose that $A\subseteq \n$ is such that $\lBD(A)>0$.  Then for any $l\in\n$,
there exist arbitrarily large $a,d\in \n$ such that the geometric sequence $G:=\{2^a(2^d)^n:n=0,1,\ldots,l-1\}$ is a $(1,1)$-approximate subset of $A$.
\end{theorem}

\begin{proof}

Fix $m\in \n$.  Since $\BD(\log A)\geq \lBD(A)>0$, the set $\log A$ contains an arithmetic progression
$\{a+nd:n=0,1,\ldots,l-1\}$ with $a,d>m$.
Fix $n\in \{0,1,\ldots,l-1\}$ and take $x\in A$ and $\theta\in [0,1)$ such that $a+nd=\log x+\theta$.  Then $x\leq 2^{a+nd}=2^{\theta}x< 2x$.
\end{proof}

The following example shows that we cannot improve upon the level of approximation in the previous theorem.

\begin{example}
For each $\epsilon>0$, there is $A\subseteq \n$ such that
$\lld(A)=\uld(A)>0$ and no positive integer power of $2$ is a
$(1-\epsilon,1)$-approximation of any element of $A$.
\end{example}

\begin{proof}

Choose $\delta>0$ such that $(2-\epsilon)2^{\delta}<2$.
Set \[A:=\bigcup_{n=1}^{\infty}[2^n+1,2^{n+\delta}].\]
Note that the interval $[2^{n+\delta},(2-\epsilon)2^{n+\delta}]$ does not contain
any positive integer power of $2$ as $2^{n+1}\leqslant (2-\epsilon)2^{n+\delta}$
implies that $2\leq (2-\epsilon)2^{\delta}$.  It follows that no power of $2$ is a $(1-\epsilon,1)$ approximation of any element of $A$.  We leave it to the reader to show that $\lld(A)=\uld(A)\geq\delta$.
\end{proof}

Our next example shows that one cannot prove Theorem \ref{geometricsequence} under the weaker assumption of positive Banach density.
\begin{example}
Let $\alpha<1$. Fix a $j$ such that $(j-1)/j>\alpha$. Let $u_0=2$,
$u_{i+1}>(ju_i)^3$, and set
\[A=\bigcup_{i=1}^{\infty}[u_i,ju_i].\]
Then $\ud(A)>\alpha$. For any $n\in\n$, there exists an $m\in\n$ such that
there does not exist $3$-term geometric progression $G=\{a,ar,ar^2\}$ with $a,r>m$
 and $G$ is an $(n-1,1)$-approximate subset of $A$.
\end{example}

For a proof of the claim in the previous example, one can consult our paper \cite{DGJLLM}.

\medskip

\noindent For $A\subseteq\n$ and $0<r\leq 1$, set
\[A^r:=\{\lceil x^r\rceil:x\in A\}.\]

One proves the following proposition in a manner similar to the proof of Proposition \ref{logA}
\begin{proposition}\label{Rreduction}
For any $A\subseteq\n$, we have $\BD(A^r)\geq\BD_r(A)$.
\end{proposition}

%
%
%

\begin{theorem}\label{powersequence}
Suppose $A\subseteq\n$ and $m\in \n$ are such that $\BD_{1/m}(A)>0$.  Then for any $\epsilon>0$ and $l\in\n$, there exist arbitrarily large $a,d\in \n$ such that $\{(a+nd)^m:n=0,1,\ldots,l-1\}$ is an $(m+\epsilon,\frac{m-1}{m})$-approximate
subset of $A$.
\end{theorem}

\begin{proof}
Fix $p\in \n$.  Since $\BD(A^{1/m})\geq \BD_m(A)>0$, there are $a,d>p$  such that $\{a+nd:n=0,1,\ldots,l-1\}\subseteq A^{1/m}$.
Choose $a$ sufficiently large so that, for any $z\geq a-1$, we have
\[\epsilon z^{m-1}>{{m}\choose{m-2}}z^{m-2}+
{{m}\choose{m-3}}x^{m-3}+\cdots+mx+1.\]
Fix $n\in \{0,1,\ldots,l-1\}$ and take $x\in A$ and $\theta \in [0,1)$ such that $a+nd=x^{1/m}+\theta$.
Since $x^{1/m}>a-1$, we have
\begin{eqnarray*}
\lefteqn{x\leqslant (a+nd)^m=(x^{\frac{1}{m}}+\theta)^m
=x+(m+\epsilon)x^{\frac{m-1}{m}}\theta}\\
& &\quad +\theta\left(-\epsilon x^{\frac{m-1}{m}}+{{m}\choose{m-2}}x^{\frac{m-2}{m}}\theta+
{{m}\choose{m-3}}x^{\frac{m-3}{m}}\theta^2
+\cdots+\theta^{m-1}\right)\\
& &<x+(m+\epsilon)x^{\frac{m-1}{m}}.
\end{eqnarray*}
Hence $(a+nd)^m$ is an $\left(m+\epsilon,\frac{m-1}{m}\right)$-approximation of $x\in A$.
\end{proof}

The next example shows that there is not much room left to improve upon the level of approximation in the previous theorem.

\begin{example}
For any $\epsilon>0$, there exists a $\delta>0$ and there exists
a set $A\subseteq\n$ such that
$\ld_{1/m}(A)=\ud_{1/m}(A)=\delta$ and such that, for any $a\in \n$, $a^m$ is not an
$\left(m-\epsilon,\frac{m-1}{m}\right)$-approximation of any element in $A$.
\end{example}

\begin{proof}

Fix $0<\delta<\epsilon/m$ and set
$\displaystyle A:=\bigcup_{n=1}^{\infty}[n^m+1,(n+\delta)^m)$.
Suppose, towards a contradiction, that $a^m\in [x,x+(m-\epsilon)x^{(m-1)/m})$ for some
$x\in [n^m+1,(n+\delta)^m)$. It follows that $n+1\leq a$, whence
$$((n+\delta)+(1-\delta))^m=(n+1)^m\leqslant(n+\delta)^m+(m-\epsilon)(n+\delta)^{m-1}.$$
Hence $m(n+\delta)^{m-1}(1-\delta)\leqslant(m-\epsilon)(n+\delta)^{m-1}$, which implies
that $\delta m\geq\epsilon$, a contradiction.  We leave it to the reader to check that $\ld_{1/m}(A)=\ud_{1/m}(A)=\delta$.
\end{proof}

\noindent Corollary \ref{positiveimpliespositive} and Theorem \ref{powersequence} immediately imply:

\begin{corollary}
Suppose that $A\subseteq\n$ is such that $\ud(A)>0$.  Then for any $l,m\in\n$ and $\epsilon>0$,
there exists arbitrarily large $a,d\in\n$ such that $\{(a+nd)^m:n=0,1,\ldots,l\}$ is an
$\left(m+\epsilon,\frac{m-1}{m}\right)$-approximate subset of $A$.
\end{corollary}

%

We should remark that the conclusions of approximate structure really are necessary.  For example, if $A$ is the set of all square-free numbers, then $\ld(A)>0$ but $A$ does not contain any $3$-term geometric progression
or any $m$-th power of an integer greater than $1$ with $m\geq 2$.

%

\end{document}